\documentclass{amsart}
\usepackage{amsfonts,amsmath,amssymb,amsthm,mathrsfs,amsxtra}
\usepackage{enumerate,verbatim}
\usepackage{lineno} 
\usepackage[all,2cell,ps]{xy}
\usepackage[pagebackref]{hyperref}
\usepackage{soul}
\usepackage{pgf,tikz}
\usetikzlibrary{arrows}
\usepackage{cleveref}

\DeclareMathOperator{\ini}{in}

\DeclareMathOperator{\reg}{reg}

\renewcommand{\ge}{\geqslant}
\renewcommand{\le}{\leqslant}
\newcommand{\<}{\left\langle}
\renewcommand{\>}{\right\rangle}

\theoremstyle{plain}
\newtheorem{theorem}{Theorem}[section]
\newtheorem{lemma}[theorem]{Lemma}
\newtheorem{proposition}[theorem]{Proposition}
\newtheorem{corollary}[theorem]{Corollary}

\theoremstyle{definition}

\newtheorem{example}[theorem]{Example}

\newtheorem{question}[theorem]{Question}

\theoremstyle{remark}
\newtheorem{remark}[theorem]{Remark}

\numberwithin{equation}{theorem}

\newcommand{\mycomment}[1]{}

\title[Products of monomial ideals related to maximal minors]
{Linear resolution of products of monomial ideals related to maximal minors}

\author[A.~Banerjee]{Arindam Banerjee}
\address{Ramakrishna Mission Vivekenanda Educational and Research Institute, Belur, West
Bengal, India}
\email{123.arindam@gmail.com}

\author[D.~Ghosh]{Dipankar Ghosh}
\address{Department of Mathematics, Indian Institute of Technology Kharagpur, West Bengal - 721302, India}
\email{dipankar@maths.iitkgp.ac.in, dipug23@gmail.com}

\author{S Selvaraja}
\address{Department of Mathematics, Indian Institute of Technology Bhubaneswar, Bhubaneswar, 752050, India}
\email{selvas@iitbbs.ac.in, selva.y2s@gmail.com}

\date{December 15, 2022}
\subjclass[2010]{Primary 13D02, 15A15}
\keywords{Monomial ideals; Determinantal ideals; Linear free resolution; Castelnuovo-Mumford regularity}

\begin{document}

\begin{abstract}
	Let $ X $ be an $ m \times n $ matrix of distinct indeterminates over a field $ K $, where $ m \le n $. Set the polynomial ring 
	$ K[X] := K[X_{ij} : 1 \le i \le m, 1 \le j \le n] $. Let $ 1 \le k < l \le n $ be such that $ l - k + 1 \ge m $. 
	Consider the submatrix $ Y_{kl} $ of consecutive columns of $ X $ from $ k $th column to $ l $th column. Let $ J_{kl} $ be the ideal 
	generated by `diagonal monomials' of all $ m \times m $ submatrices of $ Y_{kl} $, where the diagonal monomial of a square matrix means product of
	its main diagonal entries. We show that $ J_{k_1 l_1} J_{k_2 l_2} \cdots J_{k_s l_s} $ has a linear free resolution, where 
	$ k_1 \le k_2 \le \cdots \le k_s $ and $ l_1 \le l_2 \le \cdots \le l_s $. This result is a variation of a theorem due to Bruns and Conca.
	Moreover, our proof is self-contained, elementary and combinatorial.
\end{abstract}
\maketitle

\section{Introduction and Notation}
The study of Castelnuovo-Mumford regularity (or simply, regularity) of powers and products of ideals in polynomial 
rings has been a central problem in commutative algebra and algebraic geometry. One important result in this direction
was given by Cutkosky, Herzog and Trung \cite{CHT}, and independently by Kodiyalam \cite{vijay}. They proved that if $I$ is 
a homogeneous ideal of $K[x_1,\ldots, x_d]$, then the regularity of $I^s$ is asymptotically a linear function in $s$. 
This linear function behaves in the simplest possible way when $I$ is generated  in degree $m$, and all its powers have
linear free resolution, i.e., $\reg(I^s) =m \cdot s$ for all $s$.  
In general, powers of an ideal with a linear free resolution need not have linear free resolution (see \cite[Counterexample~1.10]{NP13}). 
Recently, a number of authors have been interested in classifying or identifying families of ideals whose powers and products have linear free 
resolution. For example, Fr{\"o}berg \cite{Fro90} characterized all squarefree monomial ideals generated by quadratic monomials, which have linear 
free resolution. Later, Herzog, Hibi and Zheng \cite{HHZ} proved that a monomial ideal $I$ generated in degree $ 2 $ has a linear free resolution 
if and only if every power of $ I $ has a linear free resolution. If $ I $ is a polymatroidal ideal, then all powers have linear free 
resolution, \cite[Corollary~12.6.4]{Herzog'sBook}. It was proved by Akin, Buchsbaum and Weyman \cite[Theorem~5.4]{ABW81}
that all powers of determinantal ideals of maximal minors of a generic matrix have linear free resolution.  Recently, in \cite[Theorem~5.1]{Rai18}, 
Raicu classified the determinantal ideals of a generic matrix with all powers having linear free resolution. We are interested to find a family of 
monomial ideals (related to determinantal ideals) whose powers and products have linear free resolution.

The study of determinantal ideals, rings and varieties is a classical topic in
commutative algebra, algebraic geometry and invariant theory (see \cite{BV88} and \cite{BCRV}). One important method to study determinantal ideals is to understand their initial ideals via Gr\"{o}bner basis. Throughout the article, let $ K $ be a field, and $ X $ be an $ m \times n $ matrix of distinct indeterminates over $ K $, where $ m \le n $. Set the polynomial ring $ K[X] := K[X_{ij} : 1 \le i \le m, 1 \le j \le n] $. Let $ I_t(X) $ be the ideal generated by all $ t \times t $ minors of $ X $.
The homological properties of these ideals as well as Gr\"{o}bner bases and initial ideals with respect to diagonal (or antidiagonal) monomial orders 
are well understood. 
Among these ideals of minors the best-behaved is the ideal of maximal minors, namely the ideal $ I_m(X) $.
  Let $ \tau $ be a diagonal monomial order. For example, $ \tau $ can be the lexicographic term order on $ K[X] $ induced by the order
 \[
 X_{11} > X_{12} > \cdots > X_{1n} > X_{21} > X_{22} > \cdots > X_{2n} > \cdots > X_{m1} > \cdots > X_{mn}.
 \]
It is well known that all the $ m \times m $ minors of $ X $ form a Gr\"{o}bner basis of $ I_m(X) $ with respect to $ \tau $.
 Indeed, it is proved in \cite[Theorem 0]{BZ93} and \cite[Corollary 7.6]{SZ93}, and generalized in 
 \cite[Theorem 2.1]{CNG15} that the maximal minors of $ X $ form a universal Gr\"{o}bner basis 
 (i.e., a Gr\"{o}bner basis with respect to every monomial order).  
 Regarding powers of $ I_m(X) $, in \cite[Theorem~2.1]{Con97}, Conca proved that the natural generators of 
 $ I_m(X)^s $ form a Gr\"{o}bner basis with respect to $ \tau $, and $ \ini_{\tau}(I_m(X)^s) = 
 \ini_{\tau}(I_m(X))^s $ for every $ s \ge 1 $.
See also \cite[Theorem~3.10]{BC98} for a similar result on arbitrary determinantal ideals $I_t(X)$. But, for $ m > 2 $ and $ s > 1 $, the natural generators of $ I_m(X)^s $ do not necessarily form a universal Gr\"{o}bner basis, due to Speyer and Sturmfels \cite[Corollary 5.6]{SS04}. In other words, there are monomial orders $ < $ such that $ \ini_{<}(I_m(X)^s) $ is strictly larger than $ \ini_{<}(I_m(X))^s $.

The asymptotic behavior of regularity also has been studied for powers of more than one ideals in 
\cite{BCH13}, \cite{BCon17} and \cite{Ghosh16}.
For determinantal ideals, Conca and Herzog in \cite[Theorem~6.1]{CH03} showed that products of ideals of minors of a Hankel matrix have linear free 
resolution.
Berget, Bruns and Conca \cite[Theorem~4.7]{BBC15} proved an extension of \cite[Theorem~5.4]{ABW81} to arbitrary products of the ideals $ I_t(X_t) $, 
where $ X_t $ is the submatrix of the first $ t $ rows of $ X $. They proved that $I$ and $\ini_{\tau}(I)$ have linear free resolutions, 
where $ I := I_{t_1}(X_{t_1}) \cdots I_{t_w}(X_{t_w}) $. In \cite[Theorems~1.3 and 5.3]{BC17}, Bruns and Conca  generalized this result further to a 
class of ideals that are fixed by the Borel group. They defined the {\it northeast ideals} $ I_t(a) $ of maximal minors: $ I_t(a) $ is generated by 
the $ t \times t $ minors of the $ t \times (n - a + 1) $ northeast submatrix
 \begin{equation*}\label{intro_eq}
	X_t(a) := (X_{ij} : 1 \le i \le t, \; a \le j \le n).  
 \end{equation*}
They proved that, if $ I_{t_1}(a_1), \ldots, I_{t_w}(a_w) $ are northeast ideals of maximal minors, 
and $ I := I_{t_1}(a_1) \cdots I_{t_w}(a_w) $, then $I$ and $\ini_{\tau}(I)$ have linear 
free resolutions. The aim of this article is to prove that products of certain monomial ideals related to maximal minors have linear free resolutions.

Throughout the paper, we work with the following setup.

 \vskip 1mm
 \noindent
 {\bf Setup.} Let $ 1 \le k < l \le n $ be such that $ l - k + 1 \ge m $. We denote the submatrix $ Y_{kl} $ of consecutive columns of $ X $ from $ k $th column to $ l $th column. Denote by $ [c_1,\ldots,c_m] $ the $ m \times m $ minor $ \det(X_{i c_j})_{1\le i, j \le m} $ of $ Y_{kl} $, where $ k \le c_1 < c_2 < \cdots < c_m \le l $. The diagonal of $ [c_1,\ldots,c_m] $ is defined to be the set $ \{ X_{1 c_1}, \ldots, X_{m c_m} \} $. We call the product $ X_{1 c_1} \cdots X_{m c_m} $ as a {\it diagonal monomial} of $ Y_{kl} $. Note that by a diagonal monomial, we always mean product of main diagonal entries of some maximal minor, not any arbitrary minor. Let $ I_{kl} $ be the ideal of $ K[X] $ generated by all $ m \times m $ minors of $ Y_{kl} $. Let $ J_{kl} $ be the ideal of $ K[X] $ generated by all diagonal monomials (of maximal minors) of $ Y_{kl} $; see Example~\ref{exam:submatrix-J-colon}.

 The main result of this paper is the following.
\begin{theorem}\label{thm:J-has-lin-free-reso}
	Let $ 1 \le k_1 \le k_2 \le \cdots \le k_s < n $ and $ 1 < l_1 \le l_2 \le \cdots \le l_s \le n $ be such that $ k_i < l_i $ and $ l_i - k_i + 1 \ge m $. Then, the ideal $ J := J_{k_1 l_1}  J_{k_2 l_2} \cdots J_{k_s l_s} $ has a linear free resolution.
\end{theorem}

In order to prove the main result, we study certain colon ideals related to $ J $. We show that colon ideals take very interesting forms, and they behave nicely for regularity. Using these results and various short exact sequences, we obtain the main result. 
Our proof is self-contained, elementary, and combinatorial.

Motivated by the results \cite[Theorem~2.1]{Con97} and \cite[Theorem~3.10]{BC98} and some computational evidence, we present the following question.

\begin{question}\label{ques}
	Let $ 1 \le k_1 \le k_2 \le \cdots \le k_s < n $ and 
	$ 1 < l_1 \le l_2 \le \cdots \le l_s \le n $ be such that $ k_i < l_i $ and 
	$ l_i - k_i + 1 \ge m $. Set $ I := I_{k_1 l_1}  I_{k_2 l_2} \cdots I_{k_s l_s} $ and $ J := J_{k_1 l_1}  J_{k_2 l_2} \cdots J_{k_s l_s} $. Then, is $ \ini_{\tau}(I) = J$? Moreover, do the natural generators of $ I $ form a Gr\"{o}bner basis with respect to $ \tau $?
\end{question}

Macaulay2 \cite{M2} experiments indicate that Question~\ref{ques} may have affirmative answers. The following examples are computed where it has affirmative answers.

\begin{example}
    Let $X$ be a $ 3 \times 9 $ matrix:
	\[
	\begin{bmatrix}
	X_{11} & X_{12} & X_{13} & X_{14} & X_{15} & X_{16} & X_{17}  & X_{18} & X_{19}\\
	X_{21} & X_{22} & X_{23} & X_{24} & X_{25} & X_{26} & X_{27} & X_{28} & X_{29} \\
	X_{31} & X_{32} & X_{33} & X_{34} & X_{35} & X_{36} & X_{37} & X_{38} & X_{39}
	\end{bmatrix}.
    \]
    It is verified by Macaulay2 that Question~\ref{ques} has affirmative answers if $I$ is equal to any of the following products:
    \begin{align*}
        & I_{14}I_{26}, \quad I_{15}I_{26}, \quad I_{14}^2I_{26}, \quad I_{15}I_{37}^2,\quad I_{14}I_{15}I_{26}, \quad I_{14}I_{26}I_{37}, \quad I_{14}I_{26}I_{59}, \\
        & I_{14}I_{37}I_{59}, \quad I_{14}I_{15}I_{26}I_{37}, \quad I_{14}I_{26}^2I_{37}, \quad I_{14}^2I_{26}I_{59}, \quad I_{14}I_{26}^2I_{59}, \quad I_{14}I_{26}I_{59}^2, \\
        & I_{14}^2I_{37}I_{59}, \quad I_{14}I_{37}^2I_{59}, \quad I_{14}I_{37}I_{59}^2, \quad I_{14}I_{26}I_{37}I_{59}, \quad I_{26}I_{37}^2I_{59}.
    \end{align*}
\end{example}

\begin{example}
    Let $X$ be a $ 2 \times 9 $ matrix:
	\[
	\begin{bmatrix}
	X_{11} & X_{12} & X_{13} & X_{14} & X_{15} & X_{16} & X_{17}  & X_{18} & X_{19}\\
	X_{21} & X_{22} & X_{23} & X_{24} & X_{25} & X_{26} & X_{27} & X_{28} & X_{29}
	\end{bmatrix}.
    \]
    It is verified by Macaulay2 that Question~\ref{ques} has affirmative answers if $I$ is equal to any of the following products:
    \begin{align*}
        & I_{14}I_{26}, \quad I_{15}I_{26}, \quad I_{14}^2I_{26}, \quad I_{15}I_{37}^2,\quad I_{14}I_{15}I_{26}, \quad I_{14}I_{26}I_{37}, \quad I_{14}I_{26}I_{59}, \\
        & I_{14}I_{37}I_{59}, \quad I_{14}I_{15}I_{26}I_{37}, \quad I_{14}I_{26}^2I_{37}, \quad I_{14}I_{26}^2I_{59}, \quad I_{14}I_{26}I_{37}^2, \quad I_{14}I_{15}I_{26}I_{37}I_{59}.
    \end{align*}
\end{example}

\begin{remark}
    If Question~\ref{ques} has affirmative answers, then
    \[
    \ini_{\tau}(I) = \ini_{\tau}(I_{k_1 l_1}) \ini_{\tau}(I_{k_2 l_2}) \cdots \ini_{\tau}( I_{k_s l_s}) = J.
    \]
    In this case, it follows from Theorem~\ref{thm:J-has-lin-free-reso} that $ I $ also has a linear free resolution.
\end{remark}

\begin{remark}\label{rmk:our-thm-is-new}
	Note that Theorem~\ref{thm:J-has-lin-free-reso}
	does not follow from the results of Bruns and Conca \cite[Theorems~1.3 and 5.3]{BC17}. For example, consider the 
	generic matrix $ [X_1 \; X_2 \; X_3] $. Then the product of ideals $ I_{12} := \langle X_1, X_2 \rangle =: J_{12} $
	and $ I_{23} := \langle X_2, X_3 \rangle =: J_{23} $ is $I = \langle X_1X_2, X_1X_3, X_2^2, X_2X_3 \rangle  = J$.
	The ideal $ I $ cannot be written as a product of northeast ideals of maximal minors considered in \cite{BC17}.
\end{remark}

\section{Main Result}

In order to prove Theorem~\ref{thm:J-has-lin-free-reso}, we first show that the monomial ideal $ J_{kl} $ has a linear free resolution by proving that the ideal has linear quotients. Recall that an ideal $ I $ is said to have linear quotients if there exists an ordered set of generators $ f_1,\ldots,f_r $ of $ I $ such that $ (\langle f_1,\ldots,f_u \rangle : f_{u+1}) $ is generated by linear forms for every $ 1 \le u \le r - 1 $. Before proving that $ J_{kl} $ has linear quotients, we illustrate an example for reader's convenience.

\begin{example}\label{exam:submatrix-J-colon}
	Let $X$ be a
	$ 3 \times 8 $ matrix:
	\[
	\begin{bmatrix}
	X_{11} & {\bf X_{12}} & {\bf X_{13}} & {\bf X_{14}} & {\bf X_{15}} & {\bf X_{16}} & X_{17}  & X_{18} \\
	X_{21} & {\bf X_{22}} & {\bf X_{23}} & {\bf X_{24}} & {\bf X_{25} }& {\bf X_{26}} & X_{27} & X_{28} \\
	X_{31} & {\bf X_{32}} & {\bf X_{33}} & {\bf X_{34}} & {\bf X_{35}} & {\bf X_{36}} & X_{37} & X_{38} 
	\end{bmatrix}.
	\]
	Setting $ k = 2 $ and $ l = 6 $, the submatrix $ Y_{k\,l} = Y_{2\,6} $ is given by the boldfaced entries, and the corresponding ideal $ J_{2\,6} = (f_1,f_2,\ldots,f_{10})$, where
	\begin{align*}
	&f_1 = X_{12}X_{23}X_{34}, \,f_2 = X_{12}X_{23}X_{35}, \, f_3 = X_{12}X_{23}X_{36}, \, f_4 = X_{12}X_{24}X_{35},\\
	& f_5 = X_{12}X_{24}X_{36}, f_6 = X_{12}X_{25}X_{36}, \, f_7 = X_{13}X_{24}X_{35}, \,f_8 = X_{13}X_{24}X_{36}, \\
	& f_9 = X_{13}X_{25}X_{36},\, f_{10} =  X_{14}X_{25}X_{36}.
	\end{align*}
	It can be verified by Macaulay2 \cite{M2} that
	\begin{align*}
		&(\langle f_1 \rangle : f_2) = (X_{34}), \quad\; (\langle f_1,f_2 \rangle : f_3) = (X_{34},X_{35}), \quad\; (\langle f_1,f_2,f_3 \rangle : f_4) = (X_{23}), \\
		& (\langle f_1,\ldots,f_4 \rangle : f_5) = (X_{23},X_{35}), \quad\quad (\langle f_1,\ldots,f_5 \rangle : f_6) = (X_{23},X_{24}), \\
		& (\langle f_1,\ldots,f_6 \rangle : f_7) = (X_{12}), \quad\quad (\langle f_1,\ldots,f_7 \rangle : f_8) = (X_{12},X_{35}), \\
		& (\langle f_1,\ldots,f_8 \rangle : f_9) = (X_{12},X_{24}), \quad\quad (\langle f_1,\ldots,f_9 \rangle : f_{10}) = (X_{12},X_{13}),
	\end{align*}
	which shows that $ J_{2\,6} $ has linear quotients.
\end{example}


Note that $J_{kl}$ is the ideal generated by the diagonal monomials corresponding to the maximal minors of the matrix $Y_{kl}$, thus $J_{kl} =\ini(I_{kl})$ in this case. The property ``having linear quotients" can be checked for $ J_{kl} $, equivalently, in $ K[X] $ or in $ K[Y_{kl}] $.
So in order to show that $ J_{kl} $ has linear quotients (as an ideal of $ K[X] $), it suffices to prove the following lemma.

\begin{lemma}\label{lem:lin-quo}
	Let $ \{ f_1,\ldots, f_r \} $ be the collection of all diagonal monomials of $ X $ ordered by $ \tau $, i.e., $ f_1 > \cdots > f_r $ and $ J = \langle f_1, \ldots, f_r \rangle $. Then the colon ideal $ (\langle f_1,\ldots,f_u \rangle : f_{u+1}) $ is generated by some variables for every $ 1 \le u \le r - 1 $. Indeed, if $ f_{u+1} = X_{1 c_1} X_{2 c_2} \cdots X_{m c_m} $, setting $ c_0 = 0 $, we have
	\begin{align*}
	(\langle f_1,\ldots,f_u \rangle : f_{u+1}) = \left\langle X_{1 b_1}, X_{2 b_2}, \ldots, X_{m b_m} : c_{i-1} < b_i < c_i\;\forall\; 1 \le i \le m \right\rangle.
	\end{align*}
\end{lemma}

\begin{proof}
	Fix $ i $ and $ b_i $ such that $ 1 \le i \le m $ and $ c_{i-1} < b_i < c_i \; (< c_{i+1})$. Set
	\[
		g := X_{1 c_1} \cdots X_{i-1 \, c_{i-1}} X_{i b_i} X_{i+1 \, c_{i+1}} \cdots X_{m c_m}.
	\]
	Clearly, $g$ is a diagonal monomial, and $ g > f_{u+1} $, hence $ g $ is one of $ f_1,\ldots,f_u $. Therefore
	\begin{align*}
	X_{i b_i} \in \left( \< g \> : f_{u+1} \right) \subseteq (\langle f_1,\ldots,f_u \rangle : f_{u+1}).
	\end{align*}
	This proves `the containment $ \supseteq $'.
	For another containment, it is enough to show that, for every $ 1 \le v \le u $, $ (\langle f_v \rangle : f_{u+1}) \subseteq (X_{j a_j}) $ for some $ 1 \le j \le m $ and $ c_{j-1} < a_j < c_j $. Let $ f_v = X_{1 a_1} X_{2 a_2} \cdots X_{m a_m} $. Since $ f_v $ is a diagonal monomial of $ X $, we have that $ 1 \le a_1 < \cdots < a_m \le n $. Set $ a_0 := 0 $. Since $ f_v > f_{u+1} $, there exists $ 1 \le j \le m $ such that $ a_j < c_j $ and $ a_l = c_l $ for all $ 1 \le l < j $. Note that $ c_{j-1} = a_{j-1} < a_j < c_j $. It can be easily shown that any monomial in $ (\langle f_v \rangle : f_{u+1}) $ is divisible by $ X_{j a_j} $. Hence $ (\langle f_v \rangle : f_{u+1}) \subseteq (X_{j a_j}) $, which completes the proof.
\end{proof}

As an immediate consequence, we obtain the following.

\begin{corollary}\label{cor:linear-res}
	The ideal $ J_{kl} $ of $ K[X] $ has linear quotients, and hence it has a linear free resolution.
\end{corollary}

\begin{proof}
	The corollary follows from Lemma~\ref{lem:lin-quo} and \cite[Lemma~4.1]{CH03}.
\end{proof}

The following proposition is technical and most crucial in our study. For better understanding, we refer the reader to Example~\ref{exam:for-big-lemma}, which illustrates the proof of the proposition.

\begin{proposition}\label{prop:powers-colon}
	Let $ 1 \le k_1 \le k_2 \le \cdots \le k_s < n $ and $ 1 < l_1 \le l_2 \le \cdots \le l_s \le n $ be such that $ k_i < l_i $ and $ l_i - k_i + 1 \ge m $. Suppose $ s \ge 2 $. Set $ J := J_{k_1 l_1}  J_{k_2 l_2} \cdots J_{k_s l_s} $, where $ J_{k_1 l_1} $ is generated by the diagonal monomials $ \{ f_1,\ldots, f_r \} $ of $ Y_{k_1 l_1} $ ordered by $ \tau $, i.e., $ f_1 > \cdots > f_r $. Then, for every $ 0 \le u \le r - 1 $,
    if $ f_{u+1} = X_{1 c_1} X_{2 c_2} \cdots X_{m c_m} $, setting $ c_0 = k_1 - 1 $, one has that
	\begin{align*}
		& (J+ \langle f_1,\ldots,f_u \rangle : f_{u+1}) = \\
		& J_{k_2 l_2} \cdots J_{k_s l_s} + \left\langle X_{1 b_1}, X_{2 b_2}, \ldots, X_{m b_m} : c_{i-1} < b_i < c_i \;\forall\; 1 \le i \le m \right\rangle.
	\end{align*}
\end{proposition}

\begin{proof}
	Set $ L := J_{k_2 l_2} \cdots J_{k_s l_s} + \left\langle X_{1 b_1}, X_{2 b_2}, \ldots, X_{m b_m} : c_{i-1} < b_i < c_i\;\forall\; 1 \le i \le m \right\rangle $. Since we are dealing with monomial ideals, by Lemma~\ref{lem:lin-quo}, we only need to prove that $(J : f_{u+1}) \subseteq L$. So we consider a monomial $ f \in (J : f_{u+1}) $. Hence, since $ f f_{u+1} $ is a monomial, it can be written as $ f f_{u+1} = g_1 g_2 \cdots g_s p $ for some diagonal monomials $ g_j \in J_{k_j l_j} $ $ (1 \le j \le s) $ and a monomial $ p $. Suppose $ g := \text{gcd}(f_{u+1}, (g_1 g_2 \cdots g_s)) = X_{i_1 c_{i_1}} X_{i_2 c_{i_2}} \cdots X_{i_k c_{i_k}} $, where $ i_1 < i_2 < \cdots < i_k $. Then $ f_{u+1} = g g' $ for some $ g' $ which must divide $ p $. We show that either
	\begin{align}
		& f \in \left\langle X_{1 b_1}, X_{2 b_2}, \ldots, X_{m b_m} : c_{i-1} < b_i < c_i,\; 1 \le i \le m \right\rangle \quad \mbox{ or} \label{eqn:f-in-vars}\\
		&g_1 g_2 \cdots g_s = h_1 h_2 \cdots h_s \mbox{ for some diagonal monomials } \label{eqn:g=h}\\
		& \quad \quad \quad \quad h_j \in J_{k_j l_j} \;( 1 \le j \le s ) \mbox{ such that $ h_1 $ is divisible by $ g $}. \nonumber
	\end{align}
	From \eqref{eqn:g=h}, it follows that $ f g g' = f f_{u+1} = g_1 g_2 \cdots g_s p = h_1 h_2 \cdots h_s p $, which yields that $ f = (h_2\cdots h_s)q $ for some monomial $ q $, and hence $ f \in J_{k_2 l_2} \cdots J_{k_s l_s} $. Thus, in both cases \eqref{eqn:f-in-vars} and \eqref{eqn:g=h}, $ f \in L $, which proves that $(J : f_{u+1}) \subseteq L$. Suppose \eqref{eqn:f-in-vars} does not hold true. It remains to prove \eqref{eqn:g=h}. 
	
	Each $ g_j $ is a product of $ m $ many variables. Draw circle around each of these variables in the matrix $ X $. We do the process for every $ 1 \le j \le s $, and every time we draw a new circle if there is a repetition of variable. So there might be more than one circle around a variable. Moreover, in each row, there are total $ s $ many circles listed as 1st, 2nd, 3rd etc., from left to right. (Possibly, $ j $th and $ (j+1) $st circles of a row are at the same point). For $ 1 \le j \le s $, we construct $ h_j $ as the product of $ m $ many variables corresponding to $ j $th circles of all rows of $ X $. It follows from the construction that $ g_1 g_2 \cdots g_s = h_1 h_2 \cdots h_s $. We make the following statements.
	
	{\it Claim 1.} The $ 1 $st circle in $ (i_l) $th row of $ X $ is at $ (i_l, c_{i_l}) $ for every $ 1 \le l \le k $.
	
	{\it Proof of Claim 1.} Since each $ X_{i_l\, c_{i_l}} $ divides one of $ g_j $, $ 1\le j \le s $, there is a circle around $ (i_l, c_{i_l}) $ for every $ 1 \le l \le k $. If possible, suppose that the $ 1 $st circle in $ (i_l) $th row is around $ (i_l,v) $ for some $ 1 \le l \le k $, where $ v < c_{i_l} $ and $ l $ is the minimum such possible number, i.e., the $ 1 $st circle in $ (i_j) $th row of $ X $ is around $ (i_j, c_{i_j}) $ for all $ 1 \le j < l $. Since $ (i_l,v) $ contains one circle, there is $ g_j $ for some $ 1 \le j \le s $ such that $ X_{i_l \, v} $ divides $ g_j $. We claim that there exists $ X_{i \, b_i} $ with $ i_{l-1} < i \le i_l $ and $ c_{i-1} < b_i < c_i $ such that $ X_{i \, b_i} $ divides $ g_j $. Assuming the claim, and observing that $ X_{i \, b_i} $ does not divide $ f_{u+1} $, it follows from $ f f_{u+1} = g_1 g_2 \cdots g_s p $ that $ X_{i \, b_i} $ divides $ f $, which contradicts the assumption that \eqref{eqn:f-in-vars} does not hold true. It remains to prove the claim. We may write
	\[
		g_j = X_{1 b_1} X_{2 b_2} \cdots X_{i_{l-1} b_{i_{l-1}}} \cdots X_{i_l b_{i_l}} \cdots X_{m b_m}
	\]
	for some $ b_i $, $ 1 \le i \le m $, where $ b_{i_l} = v $. Note that $ v < c_{i_l} $ and $ b_1 < b_2 < \cdots < b_m $. If possible, assume that $ b_i $ does not lie in between $ c_{i-1} $ and $ c_i $ for every $ i_{l-1} < i \le i_l $. Hence, since $ b_{i_l} = v < c_{i_l} $, it follows that $ b_{i_l} \le c_{i_l - 1} $. Then, starting from $ i_l $, by using backward induction, one obtains that $ b_i \le c_{i-1} $ for all $ i_{l-1} < i \le i_l $. In particular, $ b_{i_{l-1}} < b_{1 + i_{l-1}}  \le c_{i_{l-1}} $. Since $ X_{i_{l-1} b_{i_{l-1}}} $ divides $ g_j $, there is a circle at $ (i_{l-1}, b_{i_{l-1}}) $. So the $ 1 $st circle in $ (i_{l-1}) $th row of $ X $ does not lie at $ (i_{l-1}, c_{i_{l-1}}) $, which contradicts the minimality of $ l $. Therefore there is some $ i $ with $ i_{l-1} < i \le i_l $ for which $ b_i $ is lying in between $ c_{i-1} $ and $ c_i $. This completes the proof of Claim~1.
	
	{\it Claim 2.} For every $ 1 \le j \le s $ and every $ 1 \le i \le m - 1 $, the $ j $th circle in $ i $th row is on the left of the $ j $th circle in $ (i+1) $st row.
	
	{\it Proof of Claim 2.} If possible, suppose the $ j $th circle in some $ i $th row is either on the same column of the $ j $th circle in $ (i+1) $st row or on the right side of it. Consider $ g_v $ corresponding to each circle of the first $ j $ many circles in $ (i+1) $st row. Each such $ g_v $ corresponds to one circle in $ i $th row which is situated on the left of the $ j $th circle of same row. Therefore, in $ i $th row, there are at least $ j $ many circles before the $ j $th circle, which is a contradiction. This completes the proof of Claim~2.
	
	{\it Claim 3.} The monomial $ h_j \in J_{k_j l_j} $ for every $ 1 \le j \le s $.
	
	{\it Proof of Claim 3.} For every $ 1 \le j \le s $, we set
	\begin{align*}
		\mathcal{A}_j & := \left\{ (i,v) : 1 \le i \le m \mbox{ and } k_j \le v \le l_j \right\},\\
		\mathcal{B}_j & := \left\{ (i,v) : 1 \le i \le m \mbox{ and } 1 \le v \le l_j \right\} \quad \mbox{and}\\
		\mathcal{C}_j & := \left\{ (i,v) : 1 \le i \le m \mbox{ and } k_j \le v \le n \right\}.
	\end{align*}
	It follows from the constructions of $ \mathcal{A}_j $, $ \mathcal{B}_j $ and $ \mathcal{C}_j $ that
	\begin{equation}\label{eqn:sets-relation-A-B-C}
		\mathcal{A}_j = \mathcal{B}_j \cap \mathcal{C}_j; \quad \mathcal{B}_1 \subseteq \mathcal{B}_2 \subseteq \cdots \subseteq \mathcal{B}_s \quad \mbox{and} \quad \mathcal{C}_1 \supseteq \mathcal{C}_2 \supseteq \cdots \supseteq \mathcal{C}_s.
	\end{equation}
	Moreover, it can be observed that $ (i,v) \in \mathcal{A}_j $ if and only if the corresponding entry $ X_{iv} $ of $ X $ belongs to the submatrix $ Y_{k_j l_j} $. Therefore, for every $ 1 \le j \le s $,
	\begin{equation}\label{state-circles-gj-in-Aj}
		\mbox{the circles correspond to $ g_j \in J_{k_j l_j} $ are placed in $ \mathcal{A}_j \; (= \mathcal{B}_j \cap \mathcal{C}_j ) $}.
	\end{equation}
	Since $ h_j $ is a diagonal monomial of $ X $ (by Claim~2), in view of the construction of $ h_j $, it is enough to show that
	\begin{equation}\label{state-jth-cir-Aj}
		\mbox{the $ j $th circle of every row of $ X $ is placed in $ \mathcal{A}_j $}.
	\end{equation}
	In view of \eqref{eqn:sets-relation-A-B-C} and \eqref{state-circles-gj-in-Aj}, we obtain that the circles correspond to $ g_1, g_2, \ldots, g_j $ belong into $ \mathcal{B}_j $. Thus, in every row of $ X $, at least $ j $ many circles are there in $ \mathcal{B}_j $. Therefore the $ j $th circle of every row of $ X $ must be placed in $ \mathcal{B}_j $. Hence, in order to prove \eqref{state-jth-cir-Aj}, it remains to show that the $ j $th circle of every row of $ X $ is in $ \mathcal{C}_j $. If possible, suppose that the $ j $th circle of some $ i $th row does not belong into $ \mathcal{C}_j $. It follows that the first $ j $ circles of the $ i $th row do not belong into $ \mathcal{C}_j $. So at most $ (s - j) $ many circles of the $ i $th row can be there in $ \mathcal{C}_j $. On the other hand, in view of \eqref{eqn:sets-relation-A-B-C} and \eqref{state-circles-gj-in-Aj}, since the circles correspond to $ g_j, g_{j+1}, \ldots, g_s $ are placed in $ \mathcal{C}_j $, at least $ (s - j + 1) $ many circles of the $ i $th row are there in $ \mathcal{C}_j $, which is a contradiction. This completes the proof of \eqref{state-jth-cir-Aj}, and hence Claim~3.
	
	By Claim~3, one obtains that each $ h_j \in J_{k_j l_j} $ $ (1 \le j \le s) $ is a diagonal monomial; while Claim~1 yields that $ h_1 $ is divisible by $ g = X_{i_1 c_{i_1}} X_{i_2 c_{i_2}} \cdots X_{i_k c_{i_k}} $. This completes the proof of \eqref{eqn:g=h}, and hence the proposition.
\end{proof}

We now give an example describing the procedure that we follow in the proof of Proposition~\ref{prop:powers-colon}.

\begin{example}\label{exam:for-big-lemma}
 Let $X$ be a $6 \times 16$ matrix of distinct indeterminates over a field $K$. We describe the matrix $ X $ by the figure below, where the horizontal lines represent the rows of $ X $, the vertical lines represent the columns of $ X $, and the entries of $ X $ are situated on the intersections of horizontal and vertical lines. We follow the notations used in the proof of Proposition~\ref{prop:powers-colon}. 
 Set $(k_1,l_1)=(1,12)$, $(k_2,l_2)=(3,13)$, $(k_3,l_3)=(7,15)$, $(k_4,l_4)=(9,16)$ and $(k_5,l_5)=(10,16)$, which correspond five submatrices $ Y_{k_j l_j} $ of $ X $ and five ideals $ J_{k_j l_j} $ of $ K[X] $ for $ 1 \le j \le 5 $ respectively.
 
 \begin{tikzpicture}[scale=0.75]
\draw (1,1)-- (2,1);
\draw (2,1)-- (3,1);
\draw (3,1)-- (4,1);
\draw (4,1)-- (5,1);
\draw (5,1)-- (6,1);
\draw (6,1)-- (7,1);
\draw (7,1)-- (8,1);
\draw (8,1)-- (9,1);
\draw (9,1)-- (10,1);
\draw (10,1)-- (11,1);
\draw (11,1)-- (12,1);
\draw (12,1)-- (13,1);
\draw (13,1)-- (14,1);
\draw (14,1)-- (15,1);
\draw (1,1)-- (1,2);
\draw (2,2)-- (2,1);
\draw (1,2)-- (2,2);
\draw (3,2)-- (3,1);
\draw (4,2)-- (4,1);
\draw (5,2)-- (5,1);
\draw (6,2)-- (6,1);
\draw (7,2)-- (7,1);
\draw (8,2)-- (8,1);
\draw (9,2)-- (9,1);
\draw (10,2)-- (10,1);
\draw (11,2)-- (11,1);
\draw (12,2)-- (12,1);
\draw (13,2)-- (13,1);
\draw (14,2)-- (14,1);
\draw (15,2)-- (15,1);
\draw (2,2)-- (3,2);
\draw (3,2)-- (4,2);
\draw (4,2)-- (5,2);
\draw (5,2)-- (6,2);
\draw (6,2)-- (7,2);
\draw (7,2)-- (8,2);
\draw (8,2)-- (9,2);
\draw (9,2)-- (10,2);
\draw (10,2)-- (11,2);
\draw (11,2)-- (12,2);
\draw (12,2)-- (13,2);
\draw (13,2)-- (14,2);
\draw (14,2)-- (15,2);
\draw (14,3)-- (14,2);
\draw (15,3)-- (15,2);
\draw (13,3)-- (13,2);
\draw (12,3)-- (12,2);
\draw (11,3)-- (11,2);
\draw (10,3)-- (10,2);
\draw (9,3)-- (9,2);
\draw (8,3)-- (8,2);
\draw (7,3)-- (7,2);
\draw (6,3)-- (6,2);
\draw (5,3)-- (5,2);
\draw (4,3)-- (4,2);
\draw (3,3)-- (3,2);
\draw (2,3)-- (2,2);
\draw (1,3)-- (1,2);
\draw (1,4)-- (1,3);
\draw (1,3)-- (2,3);
\draw (2,4)-- (2,3);
\draw (2,3)-- (3,3);
\draw (3,4)-- (3,3);
\draw (3,3)-- (4,3);
\draw (4,4)-- (4,3);
\draw (4,3)-- (5,3);
\draw (5,4)-- (5,3);
\draw (6,4)-- (6,3);
\draw (5,3)-- (6,3);
\draw (7,4)-- (7,3);
\draw (6,3)-- (7,3);
\draw (8,4)-- (8,3);
\draw (7,3)-- (8,3);
\draw (9,4)-- (9,3);
\draw (8,3)-- (9,3);
\draw (10,4)-- (10,3);
\draw (9,3)-- (10,3);
\draw (11,4)-- (11,3);
\draw (10,3)-- (11,3);
\draw (12,4)-- (12,3);
\draw (11,3)-- (12,3);
\draw (13,4)-- (13,3);
\draw (12,3)-- (13,3);
\draw (14,4)-- (14,3);
\draw (13,3)-- (14,3);
\draw (15,4)-- (15,3);
\draw (14,3)-- (15,3);
\draw (1,5)-- (1,4);
\draw (2,5)-- (2,4);
\draw (1,4)-- (2,4);
\draw (3,5)-- (3,4);
\draw (2,4)-- (3,4);
\draw (4,5)-- (4,4);
\draw (3,4)-- (4,4);
\draw (5,5)-- (5,4);
\draw (4,4)-- (5,4);
\draw (6,5)-- (6,4);
\draw (5,4)-- (6,4);
\draw (7,5)-- (7,4);
\draw (6,4)-- (7,4);
\draw (8,5)-- (8,4);
\draw (7,4)-- (8,4);
\draw (9,5)-- (9,4);
\draw (8,4)-- (9,4);
\draw (10,5)-- (10,4);
\draw (9,4)-- (10,4);
\draw (11,5)-- (11,4);
\draw (10,4)-- (11,4);
\draw (12,5)-- (12,4);
\draw (11,4)-- (12,4);
\draw (13,5)-- (13,4);
\draw (12,4)-- (13,4);
\draw (14,5)-- (14,4);
\draw (13,4)-- (14,4);
\draw (13,5)-- (14,5);
\draw (15,5)-- (15,4);
\draw (14,4)-- (15,4);
\draw (14,5)-- (15,5);
\draw (12,5)-- (13,5);
\draw (11,5)-- (12,5);
\draw (10,5)-- (11,5);
\draw (9,5)-- (10,5);
\draw (8,5)-- (9,5);
\draw (7,5)-- (8,5);
\draw (6,5)-- (7,5);
\draw (5,5)-- (6,5);
\draw (4,5)-- (5,5);
\draw (3,5)-- (4,5);
\draw (2,5)-- (3,5);
\draw (1,5)-- (2,5);
\draw (16,1)-- (15,1);
\draw (16,2)-- (15,2);
\draw (16,3)-- (15,3);
\draw (16,4)-- (15,4);
\draw (16,5)-- (15,5);
\draw (16,6)-- (15,6);
\draw (14,6)-- (14,5);
\draw (13,6)-- (13,5);
\draw (12,6)-- (12,5);
\draw (11,6)-- (11,5);
\draw (10,6)-- (10,5);
\draw (9,6)-- (9,5);
\draw (8,6)-- (8,5);
\draw (7,6)-- (7,5);
\draw (6,6)-- (6,5);
\draw (5,6)-- (5,5);
\draw (4,6)-- (4,5);
\draw (3,6)-- (3,5);
\draw (2,6)-- (2,5);
\draw (1,6)-- (1,5);
\draw (1,6)-- (2,6);
\draw (2,6)-- (3,6);
\draw (3,6)-- (4,6);
\draw (4,6)-- (5,6);
\draw (5,6)-- (6,6);
\draw (6,6)-- (7,6);
\draw (7,6)-- (8,6);
\draw (8,6)-- (9,6);
\draw (9,6)-- (10,6);
\draw (10,6)-- (11,6);
\draw (11,6)-- (12,6);
\draw (12,6)-- (13,6);
\draw (13,6)-- (14,6);
\draw (14,6)-- (15,6);
\draw (15,6)-- (15,5);
\draw (16,6)-- (16,5);
\draw (16,5)-- (16,4);
\draw (16,4)-- (16,3);
\draw (16,3)-- (16,2);
\draw (16,2)-- (16,1);
\draw(3,6) circle (0.3cm);
\draw(4,6) circle (0.3cm);
\draw(7,6) circle (0.3cm);
\draw(9,6) circle (0.3cm);
\draw(10,6) circle (0.3cm);
\draw(5,5) circle (0.3cm);
\draw(6,5) circle (0.3cm);
\draw(8,5) circle (0.3cm);
\draw(10.93,4.97) circle (0.3cm);
\draw(10.93,4.97) circle (0.43cm);
\draw(5.97,3.97) circle (0.3cm);
\draw(7,4) circle (0.3cm);
\draw(9,4) circle (0.3cm);
\draw(12,4) circle (0.3cm);
\draw(13,4) circle (0.3cm);
\draw(8,3) circle (0.3cm);
\draw(9,3) circle (0.3cm);
\draw(11,3) circle (0.3cm);
\draw(13,3) circle (0.3cm);
\draw(14,3) circle (0.3cm);
\draw(9.97,1.99) circle (0.43cm);
\draw(9.97,1.99) circle (0.3cm);
\draw(13,2) circle (0.3cm);
\draw(14,2) circle (0.3cm);
\draw(15,2) circle (0.3cm);
\draw(10.9,0.94) circle (0.3cm);
\draw(12,1) circle (0.3cm);
\draw(15,1) circle (0.3cm);
\draw(15.9,0.98) circle (0.43cm);
\draw(15.9,0.98) circle (0.3cm);
\draw (3.6,6.34) node[anchor=north west] {$ \bigstar$};
\draw (4.65,5.4) node[anchor=north west] {$ \bigstar$};
\draw (5.57,4.4) node[anchor=north west] {$ \bigstar$};
\draw (7.6,3.38) node[anchor=north west] {$\bigstar$};
\draw (11.6,1.35) node[anchor=north west] {$\bigstar$};
\draw (9.45, 2.35) node[anchor=north west] {$\bigstar$};
\begin{scriptsize}
\fill [color=black] (1,1) circle (1.5pt);
\draw[color=black] (0.63,1.21) node {$6$};
\fill [color=black] (2,1) circle (1.5pt);
\fill [color=black] (3,1) circle (1.5pt);
\fill [color=black] (4,1) circle (1.5pt);
\fill [color=black] (5,1) circle (1.5pt);
\fill [color=black] (6,1) circle (1.5pt);
\fill [color=black] (7,1) circle (1.5pt);
\fill [color=black] (8,1) circle (1.5pt);
\fill [color=black] (9,1) circle (1.5pt);
\fill [color=black] (10,1) circle (1.5pt);
\draw [color=black] (11,1) ++(-3.5pt,0 pt) -- ++(3.5pt,3.5pt)--++(3.5pt,-3.5pt)--++(-3.5pt,-3.5pt)--++(-3.5pt,3.5pt);
\fill [color=black] (12,1) circle (1.5pt);
\fill [color=black] (13,1) circle (1.5pt);
\fill [color=black] (14,1) circle (1.5pt);
\draw [color=black] (15,1)-- ++(-3.5pt,-3.5pt) -- ++(7.0pt,7.0pt) ++(-7.0pt,0) -- ++(7.0pt,-7.0pt);
\fill [color=black] (1,2) circle (1.5pt);
\draw[color=black] (0.67,2.23) node {$5$};
\fill [color=black] (2,2) circle (1.5pt);
\fill [color=black] (3,2) circle (1.5pt);
\fill [color=black] (4,2) circle (1.5pt);
\fill [color=black] (5,2) circle (1.5pt);
\fill [color=black] (6,2) circle (1.5pt);
\fill [color=black] (7,2) circle (1.5pt);
\fill [color=black] (8,2) ++(-4.5pt,0 pt) -- ++(4.5pt,4.5pt)--++(4.5pt,-4.5pt)--++(-4.5pt,-4.5pt)--++(-4.5pt,4.5pt);
\fill [color=black] (9,2) circle (1.5pt);
\draw [color=black] (10,2) ++(-3.5pt,0 pt) -- ++(3.5pt,3.5pt)--++(3.5pt,-3.5pt)--++(-3.5pt,-3.5pt)--++(-3.5pt,3.5pt);
\fill [color=black] (11,2) circle (1.5pt);
\fill [color=black] (12,2) circle (1.5pt);
\draw [color=black] (13,2)-- ++(-3.5pt,-3.5pt) -- ++(7.0pt,7.0pt) ++(-7.0pt,0) -- ++(7.0pt,-7.0pt);
\fill [color=black,shift={(14,2)}] (0,0) ++(0 pt,5.25pt) -- ++(4.55pt,-7.875pt)--++(-9.09pt,0 pt) -- ++(4.55pt,7.875pt);
\fill [color=black] (15,2) circle (3.5pt);
\fill [color=black] (14,3) circle (3.5pt);
\fill [color=black] (15,3) circle (1.5pt);
\fill [color=black,shift={(13,3)}] (0,0) ++(0 pt,5.25pt) -- ++(4.55pt,-7.875pt)--++(-9.09pt,0 pt) -- ++(4.55pt,7.875pt);
\fill [color=black] (12,3) circle (1.5pt);
\draw [color=black] (11,3)-- ++(-3.5pt,-3.5pt) -- ++(7.0pt,7.0pt) ++(-7.0pt,0) -- ++(7.0pt,-7.0pt);
\fill [color=black] (10,3) circle (1.5pt);
\draw [color=black] (9,3) ++(-3.5pt,0 pt) -- ++(3.5pt,3.5pt)--++(3.5pt,-3.5pt)--++(-3.5pt,-3.5pt)--++(-3.5pt,3.5pt);
\fill [color=black] (8,3) circle (1.5pt);
\fill [color=black] (7,3) ++(-4.5pt,0 pt) -- ++(4.5pt,4.5pt)--++(4.5pt,-4.5pt)--++(-4.5pt,-4.5pt)--++(-4.5pt,4.5pt);
\fill [color=black] (6,3) circle (1.5pt);
\fill [color=black] (5,3) circle (1.5pt);
\fill [color=black] (4,3) circle (1.5pt);
\fill [color=black] (3,3) circle (1.5pt);
\fill [color=black] (2,3) circle (1.5pt);
\fill [color=black] (1,3) circle (1.5pt);
\draw[color=black] (0.67,3.27) node {$4$};
\fill [color=black] (1,4) circle (1.5pt);
\draw[color=black] (0.72,4.2) node {$3$};
\fill [color=black] (2,4) circle (1.5pt);
\fill [color=black] (3,4) circle (1.5pt);
\fill [color=black] (4,4) circle (1.5pt);
\fill [color=black] (5,4) circle (1.5pt);
\fill [color=black] (6,4) circle (1.5pt);
\draw [color=black] (7,4) ++(-3.5pt,0 pt) -- ++(3.5pt,3.5pt)--++(3.5pt,-3.5pt)--++(-3.5pt,-3.5pt)--++(-3.5pt,3.5pt);
\fill [color=black] (8,4) circle (1.5pt);
\draw [color=black] (9,4)-- ++(-3.5pt,-3.5pt) -- ++(7.0pt,7.0pt) ++(-7.0pt,0) -- ++(7.0pt,-7.0pt);
\fill [color=black] (10,4) circle (1.5pt);
\fill [color=black] (11,4) circle (1.5pt);
\fill [color=black,shift={(12,4)}] (0,0) ++(0 pt,5.25pt) -- ++(4.55pt,-7.875pt)--++(-9.09pt,0 pt) -- ++(4.55pt,7.875pt);
\fill [color=black] (13,4) circle (3.5pt);
\fill [color=black] (14,4) circle (1.5pt);
\fill [color=black] (15,4) circle (1.5pt);
\fill [color=black] (1,5) circle (1.5pt);
\draw[color=black] (0.68,5.09) node {$2$};
\draw[color=black] (0.68,6.09) node {$1$};
\fill [color=black] (2,5) circle (1.5pt);
\fill [color=black] (3,5) ++(-4.5pt,0 pt) -- ++(4.5pt,4.5pt)--++(4.5pt,-4.5pt)--++(-4.5pt,-4.5pt)--++(-4.5pt,4.5pt);
\fill [color=black] (4,5) circle (1.5pt);
\fill [color=black] (5,5) circle (1.5pt);
\draw [color=black] (6,5) ++(-3.5pt,0 pt) -- ++(3.5pt,3.5pt)--++(3.5pt,-3.5pt)--++(-3.5pt,-3.5pt)--++(-3.5pt,3.5pt);
\fill [color=black] (7,5) circle (1.5pt);
\draw [color=black] (8,5)-- ++(-3.5pt,-3.5pt) -- ++(7.0pt,7.0pt) ++(-7.0pt,0) -- ++(7.0pt,-7.0pt);
\fill [color=black] (9,5) circle (1.5pt);
\fill [color=black] (10,5) circle (1.5pt);
\fill [color=black] (11,5) circle (3.5pt);
\fill [color=black] (12,5) circle (1.5pt);
\fill [color=black] (13,5) circle (1.5pt);
\fill [color=black] (14,5) circle (1.5pt);
\fill [color=black] (15,5) circle (1.5pt);
\fill [color=black] (16,1) circle (3.5pt);
\fill [color=black] (16,2) circle (1.5pt);
\fill [color=black] (16,3) circle (1.5pt);
\fill [color=black] (16,4) circle (1.5pt);
\fill [color=black] (16,5) circle (1.5pt);
\fill [color=black] (16,6) circle (1.5pt);
\draw[color=black] (16.2,6.23) node {$16$};
\fill [color=black] (15,6) circle (1.5pt);
\draw[color=black] (15.2,6.23) node {$15$};
\fill [color=black] (14,6) circle (1.5pt);
\draw[color=black] (14.2,6.23) node {$14$};
\fill [color=black] (13,6) circle (1.5pt);
\draw[color=black] (13.2,6.23) node {$13$};
\fill [color=black] (12,6) circle (1.5pt);
\draw[color=black] (12.2,6.23) node {$12$};
\fill [color=black] (11,6) circle (1.5pt);
\draw[color=black] (11.2,6.23) node {$11$};
\fill [color=black,shift={(10,6)}] (0,0) ++(0 pt,5.25pt) -- ++(4.55pt,-7.875pt)--++(-9.09pt,0 pt) -- ++(4.55pt,7.875pt);
\draw[color=black] (10.5,6.38) node {$10$};
\fill [color=black] (9,6) circle (3.5pt);
\draw[color=black] (9.4,6.38) node {$9$};
\fill [color=black] (8,6) circle (1.5pt);
\draw[color=black] (8.2,6.23) node {$8$};
\draw [color=black] (7,6)-- ++(-3.5pt,-3.5pt) -- ++(7.0pt,7.0pt) ++(-7.0pt,0) -- ++(7.0pt,-7.0pt);
\draw[color=black] (7.4,6.38) node {$7$};
\fill [color=black] (6,6) circle (1.5pt);
\draw[color=black] (6.27,6.23) node {$6$};
\fill [color=black] (5,6) circle (1.5pt);
\draw[color=black] (5.25,6.23) node {$5$};
\draw[color=black] (4.4,6.38) node {$4$};
\fill [color=black] (5,6) circle (1.5pt);
\draw [color=black] (3,6) ++(-3.5pt,0 pt) -- ++(3.5pt,3.5pt)--++(3.5pt,-3.5pt)--++(-3.5pt,-3.5pt)--++(-3.5pt,3.5pt);
\draw[color=black] (3.4,6.38) node {$3$};
\fill [color=black] (2,6) circle (1.5pt);
\draw[color=black] (2.28,6.23) node {$2$};
\fill [color=black] (1,6) ++(-4.5pt,0 pt) -- ++(4.5pt,4.5pt)--++(4.5pt,-4.5pt)--++(-4.5pt,-4.5pt)--++(-4.5pt,4.5pt);
\draw[color=black] (1.2,6.4) node {$1$};
\fill [color=black,shift={(15.83,0.95)}] (0,0) ++(0 pt,5.25pt) -- ++(4.55pt,-7.875pt)--++(-9.09pt,0 pt) -- ++(4.55pt,7.875pt);
\fill [color=black,shift={(10.87,4.92)}] (0,0) ++(0 pt,5.25pt) -- ++(4.55pt,-7.875pt)--++(-9.09pt,0 pt) -- ++(4.55pt,7.875pt);
\fill [color=black] (5.87,3.88) ++(-4.5pt,0 pt) -- ++(4.5pt,4.5pt)--++(4.5pt,-4.5pt)--++(-4.5pt,-4.5pt)--++(-4.5pt,4.5pt);
\fill [color=black] (10.83,0.89) ++(-4.5pt,0 pt) -- ++(4.5pt,4.5pt)--++(4.5pt,-4.5pt)--++(-4.5pt,-4.5pt)--++(-4.5pt,4.5pt);
\fill [color=black] (9.9,1.99) circle (1.5pt);
\end{scriptsize}
\end{tikzpicture}

Following the proof of Proposition~\ref{prop:powers-colon}, let $f_{u+1}$, $g_1,~g_2,~g_3,~g_4$ and $g_5$ denote the diagonal monomials which are the product of entries marked as $\blacklozenge$, {$\bigstar$}, {\Large$\diamond$}, $\times$, $\bullet$ and 
$\blacktriangle$ respectively, i.e.,
\begin{align*}
f_{u+1} &= X_{11}X_{23}X_{36}X_{47}X_{58}X_{6\,11}  \hspace*{1.32cm}  g_1=X_{14}X_{25}X_{36}X_{48}X_{5\,10}X_{6\,12} \\
g_2&=X_{13}X_{26}X_{37}X_{49}X_{5\,10}X_{6\,11} \hspace*{1.2cm} 
g_3=X_{17}X_{28}X_{39}X_{4\,11}X_{5\,13}X_{6\,15} \\ 
g_4&=X_{19}X_{2\,11}X_{3\,13}X_{4\,14}X_{5\,15}X_{6\,16} \hspace*{0.7cm} 
g_5=X_{1\,10}X_{2\,11}X_{3\,12}X_{4\,13}X_{5\,14}X_{6\,16}.
\end{align*}
Set $ g := \text{gcd}(f_{u+1}, (g_1 g_2g_3g_4 g_5)) = X_{3 6} X_{6\,11} $. Since $ g $ does not divide $ g_1 $, we rearrange $ g_1,\ldots,g_5 $ as follows. Draw circle around each of the symbols {$\bigstar$}, {\Large$\diamond$}, $\times$, $\bullet$ and $\blacktriangle$ correspond to $ g_j $ for all $ 1 \le j \le 5 $. Note that there are more than one circles around the entries $ X_{2\,11} $, $ X_{5\,10} $ and $ X_{6\,16} $. In each row, there are total $ 5 $ circles listed as 1st, 2nd, 3rd etc., from left to right. For every $1 \le j \le 5$, let $h_j$ be the product of six many variables corresponding to $ j $th circles of all six rows of $X$, i.e.,
\begin{align*}
h_1&=X_{13}X_{25}X_{36}X_{48}X_{5\,10}X_{6\,11}  \hspace*{1.35cm}  h_2=X_{14}X_{26}X_{37}X_{49}X_{5\,10}X_{6\,12} \\
h_3&=X_{17}X_{28}X_{39}X_{4\,11}X_{5\,13}X_{6\,15} \hspace*{1.1cm} 
h_4=X_{19}X_{2\,11}X_{3\,12}X_{4\,13}X_{5\,14}X_{6\,16} \\  h_5&=X_{1\,10}X_{2\,11}X_{3\,13}X_{4\,14}X_{5\,15}X_{6\,16}.
\end{align*}
Clearly, $ g_1\cdots g_5 = h_1\cdots h_5 $. The first circles of 3rd and 6th rows of $ X $ are at $ (3,6) $ and $ (6,11) $ respectively as shown in {\it Claim~1}. Hence $ g = X_{3 6} X_{6\,11} $ divides $ h_1 $. Moreover, $ h_j $ ($ 1 \le j \le 5 $) are diagonal monomials of $ X $ as shown in {\it Claim~2}. Also $ h_1 \in J_{1\,12}$, $h_2 \in J_{3\,13}$, $h_3 \in J_{7\,15}$, $h_4 \in J_{9\,16}$ and $h_5 \in J_{10\,16}$ as shown in {\it Claim~3}.
\end{example}

In Proposition~\ref{prop:powers-colon}, it is important that the order of the ideals in the product and that the $k_i$ and the $l_i$ are in increasing order.

\begin{remark}\label{rmk:colon1}
    Let $ X $ be a $ 3 \times 9 $ generic matrix.
    \begin{enumerate}[(1)]
    \item 
    Consider the ideals $ J_{15} $ and $ J_{37} $ in $K[X]$. In this case,
    \begin{center}
        $ J_{37} = (X_{13} X_{24} X_{35}, X_{13} X_{24} X_{36}, \ldots) $ and $ (J_{37} J_{15} : X_{13} X_{24} X_{35}) \neq J_{15} $.
    \end{center}
    It ensures that Proposition~\ref{prop:powers-colon} is not necessarily true for different order of the ideals $ J_{k_1 l_1},  J_{k_2 l_2}, \ldots, J_{k_s l_s} $.
    \item 
    Consider $J_{28}$ and $J_{37}$. Writing
	\[
		J_{28} = (X_{12} X_{23} X_{34}, X_{12} X_{23} X_{35}, \ldots, X_{14} X_{25} X_{38}, X_{14} X_{26} X_{37},\ldots,X_{16} X_{27} X_{38}),
	\]
	we note that
	\[
		(J_{28} J_{37}, X_{12} X_{23} X_{34}, \ldots,X_{14} X_{25} X_{38}  : X_{14} X_{26} X_{37}) \neq J_{37} + (X_{12}, X_{13}, X_{25}).
	\]
    Thus, in Proposition~\ref{prop:powers-colon}, the orders $ k_1 \le k_2 \le \cdots \le k_s $ and $ l_1 \le l_2 \le \cdots \le l_s $ cannot be omitted.
    \end{enumerate}
\end{remark}

We recall the following well-known facts for later use.

\begin{remark}
	Let $ I $ be a homogeneous ideal of $ K[X] $. Then
	\begin{enumerate}[(i)]
		\item $ \reg(I) = \reg(K[X]/I) + 1 $.
		\item Suppose $ I $ is generated by homogeneous elements all of degree $ d $. Then $ \reg(I) \ge d $, and equality occurs if and only if $ I $  has a linear free resolution.
		\item If $ I $ is generated by some variables, then $ \reg(K[X]/I) = 0 $.
	\end{enumerate}
\end{remark}

Now we are in a position to prove our main result.

\begin{proof}[Proof of Theorem~\ref{thm:J-has-lin-free-reso}]
	We prove the assertion by induction on $ s $. By virtue of Corollary \ref{cor:linear-res}, $ J_{k_1 l_1} $ has a linear free resolution. Thus, if $ s = 1 $, then there is nothing to prove. So we may assume that $ s \ge 2 $, and $ J' := J_{k_2 l_2} J_{k_3 l_3} \cdots J_{k_s l_s} $ has a linear free resolution, i.e., $ \reg(J') = (s-1)m $. Setting $ J_{k_1 l_1} = (f_1,\ldots,f_r) $ as in Proposition~\ref{prop:powers-colon}, we consider the following short exact sequences:
	\begin{align*}
	0 \longrightarrow \dfrac{K[X]}{(J : f_1)}(-m) \stackrel{f_1}{\longrightarrow} & \dfrac{K[X]}{J} \longrightarrow	\dfrac{K[X]}{(J , f_1)} \longrightarrow 0, \\
	0 \longrightarrow \dfrac{K[X]}{(\langle J,f_1\rangle : f_2)}(-m)	\stackrel{f_2}{\longrightarrow} & \dfrac{K[X]}{(J , f_1)} \longrightarrow	\dfrac{K[X]}{(J , f_1, f_2)} \longrightarrow 0, \\
	& \vdots \\
	0 \longrightarrow \dfrac{K[X]}{(\langle J,	f_1,\ldots,f_{r-1}\rangle : f_r)}(-m)	\stackrel{f_r}{\longrightarrow}	& \dfrac{K[X]}{(J, f_1,\ldots,f_{r-1})} \longrightarrow \dfrac{K[X]}{J_{k_1 l_1}} \longrightarrow 0.
	\end{align*}
 It follows from these short exact sequences that
 \begin{align*}
 	& \reg\left( K[X]/J \right) \le \\
 	& \max \left\{ \reg\left(\dfrac{K[X]}{(J :f_1)}\right)+m, \ldots, 
 	\reg\left(\dfrac{K[X]}{(\langle J,f_1,\ldots,f_{r-1}\rangle :f_r)}\right) + m,~ 
 	\reg\left(\dfrac{K[X]}{J_{k_1 l_1}}\right) \right\}.
 \end{align*}
By Corollary~\ref{cor:linear-res}, one obtains that $ \reg\left( K[X]/J_{k_1 l_1} \right) = m - 1 $. Moreover, for every $ 1 \le i \le r $, we have
\begin{align*}
	& \reg\left( \dfrac{K[X]}{(\langle J,f_1,\ldots,f_{i-1} \rangle : f_i)} \right) + m \\
	= & \;\reg\left( \dfrac{K[X]}{\langle J'+\text{(some variables)} \rangle } \right) + m  \quad \mbox{[by Proposition~\ref{prop:powers-colon}]} \\
	\le & \;\reg\left( K[X]/J' \right) + m \quad \quad \mbox{[by \cite[Corollary~3.2(b)]{Her07}]} \\
	= & \;  sm - 1 \quad \mbox{[being $ \reg\left( K[X]/J' \right) = \reg(J') - 1 = (s-1)m - 1 $]}.
\end{align*}
Thus it follows from the above inequalities that $ \reg(K[X]/J) \le sm - 1 $, i.e., $ \reg(J) \le sm $, and hence $ \reg(J) = sm $. So $ J $ has a linear free resolution.
%
\end{proof}

%
\bibliographystyle{abbrv}  
\bibliography{refs_reg}

\begin{thebibliography}{AAAA}
	
	
	\bibitem{ABW81} K.~Akin, D.~A.~Buchsbaum and J.~Weyman, {\it Resolutions of determinantal ideals: the submaximal minors}, Adv. Math. {\bf 39} (1981), 1--30.
%
	
	\bibitem{BBC15} A.~Berget, W.~Bruns and A.~Conca, {\it Ideals generated by superstandard tableaux}, in: D. Eisenbud, et al. (Eds.), Commutative Algebra and Noncommutative Algebraic Geometry, in: MSRI Publications, vol.67, Cambridge University Press, 2015, pp.43--62.
	
	\bibitem[BZ93]{BZ93} D.~Bernstein and A.~Zelevinsky, {\it Combinatorics of maximal minors}, J. Algebraic Combin. {\bf 2} (1993) 111--121.
	

	\bibitem[BC17a]{BC17} W.~Bruns and A.~Conca, {\it Products of Borel fixed ideals of maximal minors}, 
	Adv. in Appl. Math. {\bf 91} (2017), 1--23.
	
         \bibitem[BC17b]{BC171} W.~Bruns and A.~Conca, {\it Linear resolutions of powers and products},
	Singularities and computer algebra, pages 47–69. Springer, Cham, 2017.
	\bibitem[BV88]{BV88} W.~Bruns and U.~Vetter, {\it Determinantal rings}, volume 1327 of Lecture Notes in 
	Mathematics. Springer-Verlag, Berlin, 1988.
	\bibitem[Con97]{Con97} A.~Conca, {\it Gr\"{o}bner bases of powers of ideals of maximal minors}, J. Pure Appl. Algebra {\bf 121} (1997), 223--231.
	

	\bibitem[CNG15]{CNG15} A.~Conca, E.~De Negri, E.~Gorla, {\it Universal Gr\"{o}bner bases for maximal minors}, Int. Math. Res. Not. {\bf 11} (2015) 3245--3262.
	
      
	\bibitem[Her07]{Her07} J. Herzog, {\it A generalization of the Taylor complex construction},         Comm. Algebra. {\bf 5} (2007) 1747--1756.
	
%
%
%
%
	
	\bibitem[Rai18]{Rai18} C.~Raicu, {\it Regularity and cohomology of determinantal thickenings}, Proc. Lond. Math. Soc. (3) {\bf 116} (2018), 248--280.
	
	\bibitem[SS04]{SS04} D.~Speyer and B.~Sturmfels, {\it The tropical Grassmannian}, Adv. Geom. {\bf 4} (2004) 389--411.
	
	\bibitem[SZ93]{SZ93} B.~Sturmfels and A.~Zelevinsky, {\it Maximal minors and their leading terms}, Adv. Math. {\bf 98} (1993) 65--112.
%
		
\end{thebibliography}
\end{document}